\documentclass[final,11pt]{article}
\usepackage{amsthm}
\usepackage{cite}
\usepackage{amsmath,amssymb,amsfonts}
\usepackage[dvipsnames,usenames]{color}

\usepackage{hyperref}
\usepackage[bitstream-charter]{mathdesign}
\usepackage[utf8]{inputenc}
\usepackage{tikz}
\usepackage{caption}
\usepackage{subcaption}
\usepackage{enumerate}
\usepackage{geometry}
\usepackage[affil-it]{authblk}
\usepackage{tikz-cd}
\usepackage{titling}

\newtheorem{thm}{Theorem}[section]
\newtheorem{prop}[thm]{Proposition}
\newtheorem{lem}[thm]{Lemma}
\newtheorem{defi}[thm]{Definition}

\newtheorem{cor}[thm]{Corollary}
\newtheorem{claim}[thm]{Claim}

\newcommand{\covectors}{\ensuremath{\mathcal{L}}}

\newcommand{\M}{\ensuremath{\mathcal{M}}}

\newcommand{\R}{\ensuremath{\mathbb{R}}}

\newcommand{\cross}{\ensuremath{{\rm cross}}}
\newcommand{\osc}{\ensuremath{{\rm osc}}}
\newcommand{\vc}{\ensuremath{{\rm vc}}}
\newcommand{\C}{\ensuremath{{\mathcal{C}}}}
\newcommand{\RS}{\ensuremath{{\rm RS}}}
\newcommand{\Sep}{\ensuremath{{\rm Sep}}}
\renewcommand{\Im}{\ensuremath{{\rm Im}}}
\renewcommand{\L}{\ensuremath{{\mathcal{L}}}}
\newcommand{\T}{\ensuremath{{\mathcal{T}}}}
\newcommand{\A}{\ensuremath{{\mathcal{A}}}}
\renewcommand{\H}{\ensuremath{{\mathcal{H}}}}
\renewcommand{\P}{\ensuremath{{\mathcal{P}}}}
\newcommand{\rank}{\ensuremath{{\rm rank}}}

\begin{document}
\thanksmarkseries{arabic}
\title{Unlabeled sample compression schemes for oriented matroids}
\author{Tilen Marc\thanks{Electronic address: \texttt{tilen.marc@fmf.uni-lj.si}}}
\affil{Faculty of Mathematics and Physics, Ljubljana, Slovenia \\
Institute of Mathematics, Physics, and Mechanics, Ljubljana, Slovenia}

\sloppy
\maketitle

\begin{abstract} 
A long-standing sample compression conjecture asks to linearly bound the size of the optimal sample compression schemes  by the Vapnik-Chervonenkis (VC) dimension of an arbitrary class. In this paper, we explore the rich metric and combinatorial structure of oriented matroids (OMs) to construct proper unlabeled sample compression schemes for the classes of topes of OMs bounded by their VC-dimension.
  The result extends to the topes of affine OMs, as well as to the topes of the complexes of OMs that possess a corner peeling. The main tool that we  use are the solutions of certain oriented matroid programs. 
\end{abstract}

\section{Introduction}\label{sec:intro}

The sample compression schemes were introduced in \cite{littlestone1986relating} as a generalization of the underlying structure of statistical learning algorithms. Their aim is to explore how labeled samples can be compressed, while still being able to reverse (reconstruct) the map. Depending on the sample space, the challenge is to construct a compression scheme of minimal size.

Let $U$ be a finite set, usually called the \emph{universe}, and $\C$ some family of subsets of $U$, sometimes referred to as a \emph{concept class}. It is convenient to view $\C$ as a set of $\{+, -\}$-vectors, i.e.~$\C\subseteq \{+,-\}^U$, and we shall denote with $c_u$ the $u$-th coordinate of $c\in\C$.
A \emph{sample} $s$ is simply an element of $\{+,-,0\}^U$, and its \emph{support} is $\underline{s} = \{u \in U \mid s_u \neq 0\}$. We say that a sample $s$ is \emph{realized} by $\C$ if $s \leq c$ for some $c\in \C$, where $\leq$ denotes the product ordering of elements of $\{+,-,0\}^U$ relative to the ordering $0 \leq +, 0 \leq -$.
For a concept class $\C$, let
$\RS(\C)$ be the set of all samples realizable by $\C$.

\begin{defi}
An unlabeled sample compression scheme of size $k$, for a concept class $\C \subseteq \{+,-\}^U	$, is defined by a (compressor) function 
$$\alpha: \RS(\C) \rightarrow  {U \choose \leq k} $$
and a (reconstructor) function
$$\beta : \alpha(\RS(\C)) \rightarrow \{-, +\}^U,$$
such that for any realizable sample $s\in \RS(\C)$ of $\C$, the following conditions hold: $\alpha(s) \subseteq \underline{s}$ and $s \leq \beta(\alpha(s))$.
\end{defi}
If the reconstructor function maps into $\C$, the compression scheme is known as \emph{proper}, otherwise we call it \emph{improper} \cite{chepoi2021labeled}.

\vspace{7pt}
\noindent
\textit{Example:} Let $U$ be a set of points in $\R^d$. A typical task in machine learning would be to classify the points in $U$ into two groups, based on some properties of the entities these points are describing. Assume that a linear classifier is used, i.e. the points are separated by a linear function (a hyperplane). For the sake of simplicity, let $U=\{p_1, p_2, p_3, p_4\}$, for $p_i\in \R^2$, and assume that all the points are lying on a same line. As Figure \ref{fig:class1} indicates, the space of all possible ways to classify the points with a linear function can be represented by a concept class $\C = \{++++, +++-, ++--, +---, ----, ---+, --++, -+++\}$, while the set of all samples realized by $\C$ is bigger, see Figure \ref{fig:class2} and Table \ref{tab:1}. 
\begin{figure}
\begin{subfigure}{.5\textwidth}
\centering
\includegraphics[scale=0.16]{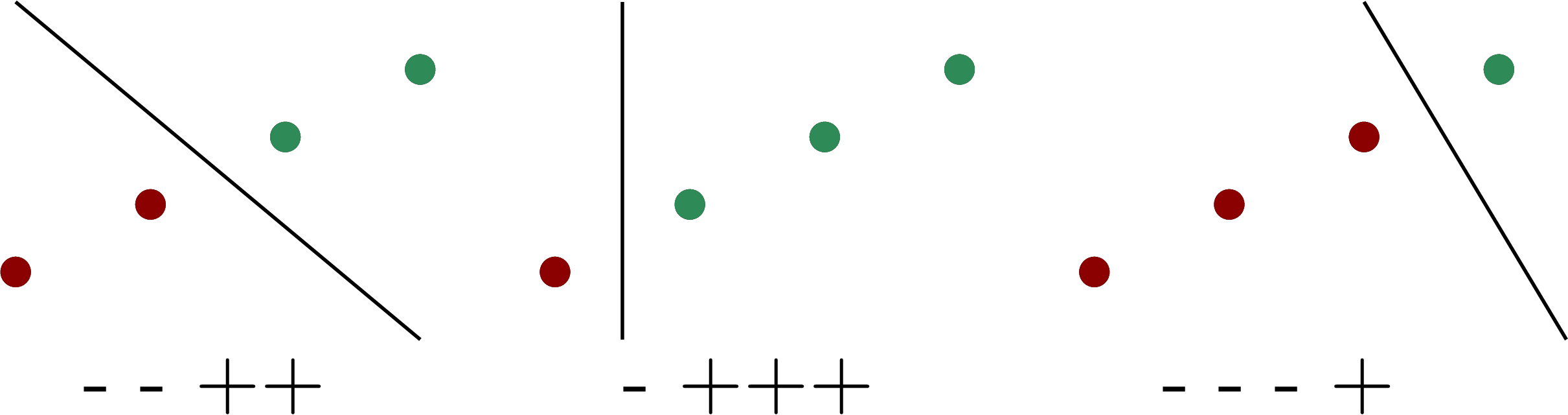}
\caption{Concepts}
\label{fig:class1}
\end{subfigure}
\begin{subfigure}{.5\textwidth}
\centering
\includegraphics[scale=0.16]{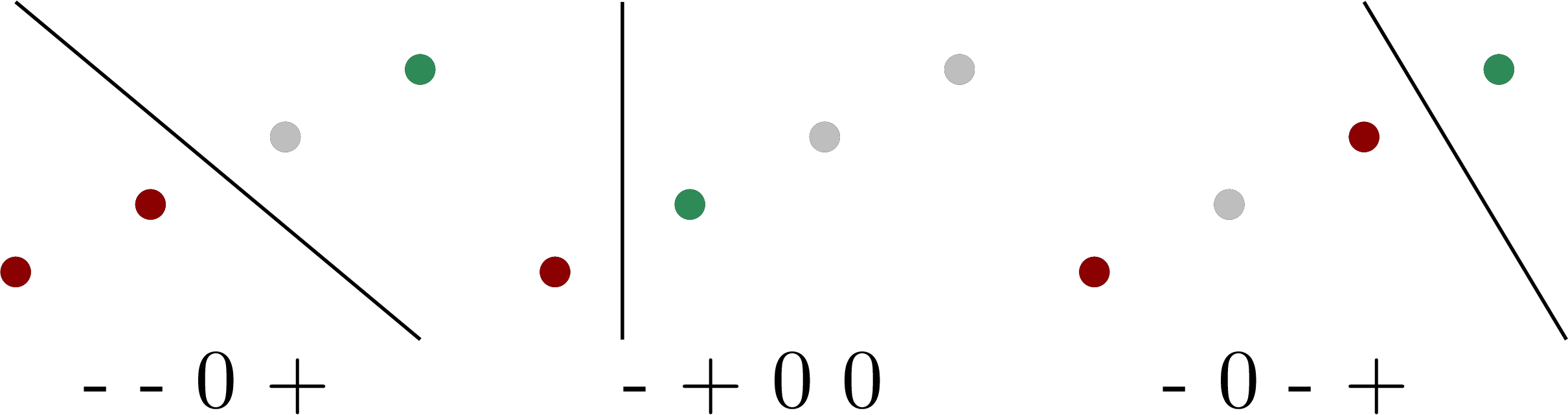}
\caption{Realizable samples}
\label{fig:class2}
\end{subfigure}
\caption{Concepts and samples coming from linear classifiers}
\label{fig:class}
\end{figure}
It seems that the representation of the samples as elements of $\{+,-,0\}^U$ is suboptimal, which would be even more obvious if $U$ had much more than four points. The problem that this paper addresses can be informally stated as: how much can we compress the samples? In this example, each sample in  $\RS(\C)$ comes from a linear function separating the points in $U$, see Figure \ref{fig:class2}, hence we can use these classifiers to compress the samples. For example, one could define a compression scheme that compresses the sample $-0-+$ to two numbers $\{3,4\}$ (and similarly for other samples), since the sample origins from a line separating points $p_3$ and $p_4$ (this approach could be generalized to an arbitrary number of points on a line). In the Table \ref{tab:1} below, we give a proper unlabeled sample compression scheme of size 2 for this simple $\C$, which was derived from the results of this paper. As we will see, $\C$ can be geometrically seen as a set of topes of an oriented matroid, or as vertices of an 8-cycle embedded into a hypercube $\{+, - \}^U$, see Figure \ref{fig:1}. The results of our paper apply to the generalizations of this example, where $U$ can be arbitrarily big with points arbitrarily placed in $\R^d$.

%

{
\begin{table}[hb]
\fontsize{8pt}{9pt}\selectfont
$$\alpha:
\begin{array}{ccccccccccccccccccc}
++++ & \mapsto &  \{1,4\} &  & +++- & \mapsto &  \{3,4\} & & ++-- & \mapsto &  \{2, 3\} &  & +--- & \mapsto &  \{\} &  & ---- & \mapsto &  \{1\} \\

---+ & \mapsto &  \{4\} &  & --++ & \mapsto & \{3\} & & -+++ & \mapsto & \{2\} &  & +++0 & \mapsto & \{1,3\} &  & ++-0 & \mapsto & \{2,3\} \\

+--0 & \mapsto & \{\} &  & ---0 & \mapsto & \{1\} & & --+0 & \mapsto & \{3\} &  & -++0 & \mapsto & \{2\} &  & ++0+ & \mapsto & \{1,4\} \\

++0- & \mapsto & \{2,4\} &  & +-0- & \mapsto & \{\} & & --0- & \mapsto & \{1\} &  & --0+ & \mapsto & \{4\} &  & -+0+ & \mapsto & \{2\} \\

+0++ & \mapsto & \{1,4\} &  & +0+- & \mapsto & \{3,4\} & & +0-- & \mapsto & \{\} &  & -0-- & \mapsto & \{1\} &  & -0-+ & \mapsto & \{4\} \\

-0++ & \mapsto & \{3\} &  & 0+++ & \mapsto & \{2\} & & 0++- & \mapsto & \{3,4\} & & 0+-- & \mapsto & \{2,3\} &    & 0--- & \mapsto & \{\} \\

0--+ & \mapsto & \{4\} &  & 0-++ & \mapsto & \{3\} & & ++00 & \mapsto & \{1,2\} & & +-00 & \mapsto & \{\} & & --00 & \mapsto & \{1\}  \\

-+00 & \mapsto & \{2\} &  & +0+0 & \mapsto & \{1,3\} & & +0-0 & \mapsto & \{\} & & -0-0 & \mapsto & \{1\} & & -0+0 & \mapsto & \{3\}  \\

0++0 & \mapsto & \{2\} &  & 0+-0 & \mapsto & \{2,3\}  & & 0--0 & \mapsto & \{\} & & 0-+0 & \mapsto & \{3\} & & +00+ & \mapsto & \{1,4\}  \\

+00- & \mapsto & \{\} &  & -00- & \mapsto & \{1\}  & & -00+ & \mapsto & \{4\} & & 0+0+ & \mapsto & \{2\} & & 0+0- & \mapsto & \{2,4\}  \\

0-0- & \mapsto & \{\} &  & 0-0+ & \mapsto & \{4\}  & & 00++ & \mapsto & \{3\} & & 00+- & \mapsto & \{3,4\} & & 00-- & \mapsto & \{\}  \\

00-+ & \mapsto & \{4\} &  & +000 & \mapsto & \{\}  & & -000 & \mapsto & \{1\} & & 0+00 & \mapsto & \{2\} & & 0-00 & \mapsto & \{\} \\

00+0 & \mapsto & \{3\} &  & 00-0 & \mapsto & \{\}  & & 000+ & \mapsto & \{4\} & & 000- & \mapsto & \{\} & & 0000 & \mapsto & \{\} \\

\end{array}
$$

$$\beta:
\begin{array}{ccccccccccccccc}
\{\} & \mapsto &  +--- & & \{3\} & \mapsto &  --++ &  & \{1,3\} & \mapsto & ++++ &  & \{2,4\} & \mapsto  &  ++-- \\
\{1\} & \mapsto &  ---- & & \{4\} & \mapsto & ---+ & & \{1,4\} & \mapsto & ++++ &  & \{3,4\} & \mapsto & +++- \\
\{2\} & \mapsto &  -+++ &  & \{1,2\}  & \mapsto & ++++ & & \{2, 3\} & \mapsto  & ++-- &  &  &  &  \\
%
\end{array}
$$

\caption{An example of a proper unlabeled compression scheme of size 2.}
\label{tab:1}
\end{table}
}

We note that \emph{labeled} sample compression schemes (which are not the topic of this paper) map into subsamples instead of ${U \choose \leq k}$, i.e. $\alpha: s \mapsto s'$ such that $s' \leq s$. Any unlabelled compression scheme $\alpha, \beta$ yields a labelled one $\alpha', \beta'$ of the same size in the following way: For any $s\in \RS(\C)$ define $\alpha'(s)$ as the subsample of $s$ with the support exactly $\alpha(s)$, while let $\beta'(\alpha'(s)) = \beta(\alpha(s)) = \beta(\underline{\alpha'(s)})$.
Hence, the unlabelled case is harder.

Consider again $\C$ as a family of subsets of $U$. A subset $X$ of $U$ is \emph{shattered} by $\C$ if for all $Y \subseteq X$ there exists $S \in \C$ such that $S \cap X = Y$. We will denote by $\overline{X}(\C)$ the family of all the subsets of $U$ that are shattered by $\C$. The \emph{Vapnik-Chervonenkis dimension} (VC-dimension)  $\vc(\C)$ of $\C$ is the cardinality of the largest subset of $U$ shattered by $\C$ \cite{vapnik2015uniform}. This well-established measurement can be considered as a complexity measure of a set system. 

A long-standing conjecture of \cite{floyd1995sample} is asking if any set family $\C$ of VC-dimension $d$ has a sample compression scheme of size $O(d)$. The investigation of geometrically structured concept classes lead to many surprising results, and strong bounds on the sizes of their compression schemes were derived \cite{ben1998combinatorial, helmbold1990learning, chalopin2022unlabeled, moran2016labeled, rubinstein2012geometric}. In particular, compression schemes for set families such as maximum concept classes, ample sets, intersection-closed concept classes, etc., were designed.  Deep connections between the structure of metric set families and compression schemes were established. 

Recently, a focus was turned towards Oriented Matroids (OMs) \cite{bjvestwhzi-93, chepoi2022ample, chepoi2021labeled}. These well-studied structures provide a common generalization and a framework for studying properties of many geometric objects such as hyperplane  arrangements, linear programming, convex polytopes, directed graphs, neural codes, etc. Thanks to the Topological Representation Theorem \cite[Theorem 5.2.1]{bjvestwhzi-93}, OMs can be represented as arrangements of pseudo-spheres giving a deep geometrical insight into their structure. Furthermore, OMs can be determined by their topes (see Figure \ref{fig:1}), which can be seen as concept classes. The Complexes of Oriented Matroids (COMs) \cite{Ban-18} further generalize the structure of OMs. 

In \cite{chepoi2022ample}, authors tackled the problem of constructing sample compression schemes by extending a concept class to another one, known to have a (small) sample compression scheme. Their results imply the existence of improper labeled schemes with size linearly bounded by its VC-dimension for certain COMs, as well as improper unlabeled schemes for OMs. 
An alternative, more direct approach was given in \cite{chepoi2021labeled}, constructing proper labeled compression schemes for COMs (and hence also OMs) with size bounded by their VC-dimension. The same bound on the unlabeled sample compression schemes for realizable (by hyperplane arrangements) affine OMs was given in \cite{ben1998combinatorial}.

In the present paper, we extend this line of work by providing proper unlabeled compression schemes for OMs of size bounded by their VC-dimension. We extend this result to COMs that possess corner peelings, answering a question from \cite{chepoi2021labeled}. Similarly to \cite{chalopin2022unlabeled}, we construct the schemes geometrically, where in our case the main tool are solutions to oriented matroid programs.

\section{Preliminary definitions and first results}\label{sec:prelim}

\subsection{OMs and partial cubes}
\noindent
\textbf{OMs: }
We present some basic definitions and facts from the well-established and rich theory of OMs, see \cite{bjvestwhzi-93} for an exhaustive coverage of the topic. For a finite set $U$ of \emph{elements}, we call $\L \subseteq \{+,-,0\}^U$ a \emph{system of sign vectors}. In the theory of OMs, elements of $\L$ are usually called \emph{covectors}.
For $X, Y \in \L$, the \emph{separator} of $X$ and $Y$ is  $\Sep(X, Y) = \{e \in U: X_e = +,Y_e = - \text{ or } X_e = -, Y_e = +\}$. The composition of $X$ and $Y$ is the sign vector $X \circ Y$, where $(X \circ Y )_e = X_e$ if $X_e \neq 0$ and $(X \circ Y )_e = Y_e$ if $X_e = 0$. As in the case of samples, the \emph{support} of a covector $X \in \{+,-,0\}^U$ is $\underline{X} = \{e\in U \mid X_e \neq 0\}  \subseteq U$.

\begin{defi}
An \emph{oriented matroid} is a system of sign vectors $\M=(U,\L)$ satisfying
\begin{itemize}
\item[\textbf{(C)}] $X\circ Y \in  \covectors$  for all $X,Y \in  \covectors$.
\item[\textbf{(SE)}]  for each pair $X,Y\in\covectors$ and for each $e\in  \Sep(X,Y)$ there exists $Z \in  \covectors$ such that
$Z_e=0$  and  $Z_f=(X\circ Y)_f$  for all $f\in U \setminus \Sep(X,Y)$.
\item[\textbf{(Sym)}] $-X\in\covectors$ for all $X\in\covectors$.
\end{itemize}
\end{defi}

We only consider simple systems of sign vectors $\L$, that is, if for each $e \in U , \{X_e : X \in \L\} =
\{+,-,0\}$ (no element is constant) and for all $e \neq f$ there exist $X, Y \in \L$ with  $X_e=X_f$, $Y_e \neq Y_f$ and  $X_e,X_f,Y_e,Y_f \in \{+, -\}$ (no two elements are parallel). The extension of the results developed in this paper to non-simple OMs is trivial.

A common source of OMs comes from central hyperplane arrangements in $\R^d$, see Figure \ref{fig:1}. Covectors $\L$ correspond to all possible  positions of points in the space with respect to the hyperplanes. Hence one can make a correspondence between regions of $\R^d$ and covectors. In fact, this can be generalized, since the Topological Representation Theorem \cite[Theorem 5.2.1]{bjvestwhzi-93} allows to represent all OMs as arrangements of pseudo-spheres in $S^{d-1}$.  We do not need to use the representation theorem directly to derive the results in our paper (hence we will avoid a precise definition), but it helps to illustrate them, as in Figure \ref{fig:1}. 

\begin{figure}
\centering
\includegraphics[scale=0.15]{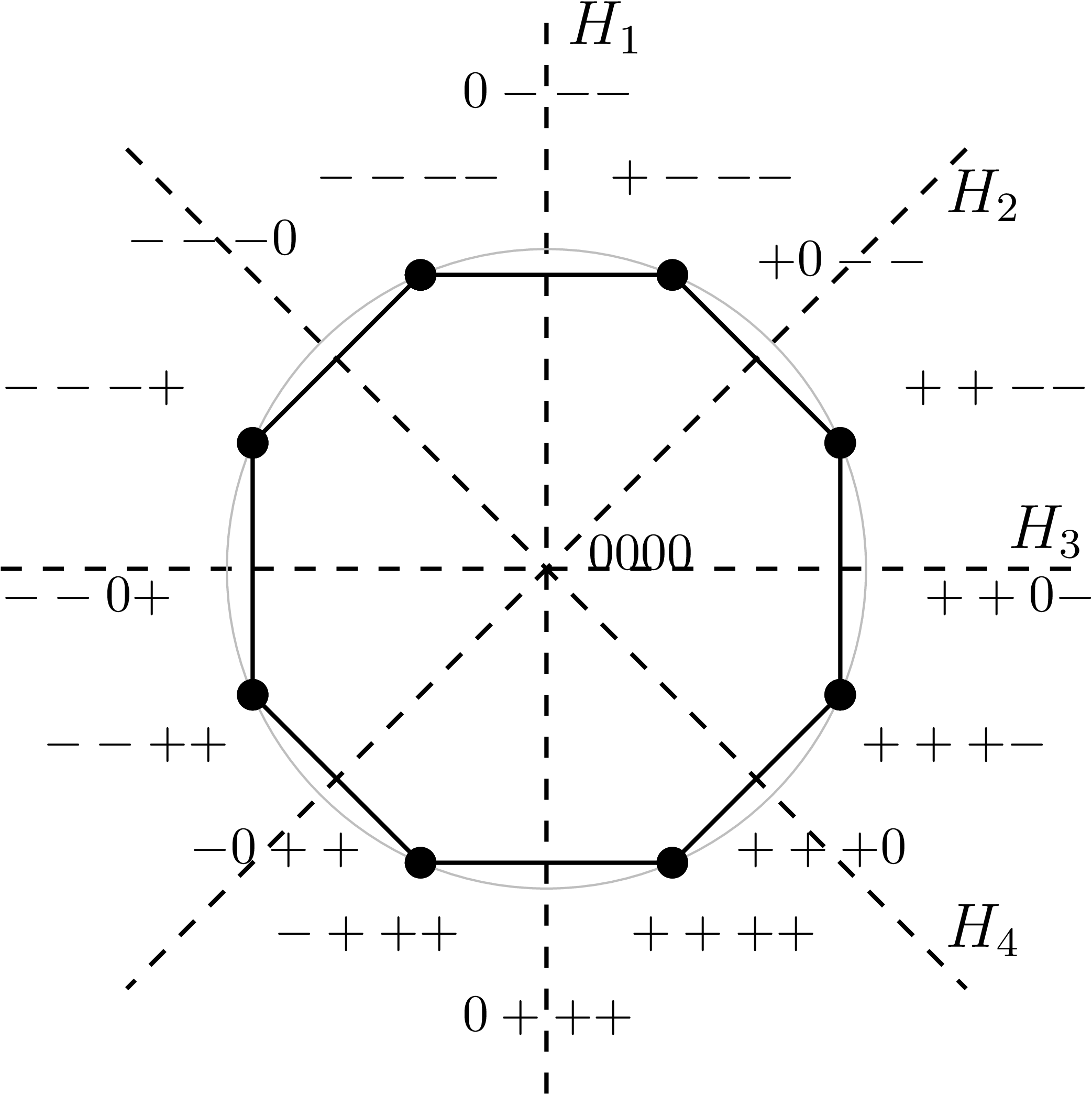}
\caption{Correspondence between sign vectors (covectors) of an OM $\M$ and its topological representation. The tope graph of $\M$ can be seen as a graph embedded in a hypercube with vertices $\{+,-\}^n$.}
\label{fig:1}
\end{figure}

We will use two classic operations on OMs. For $\M=(U,\L)$ with $e \in U$, define the \emph{deletion} as an OM  $\M\backslash e=(U\backslash \{e\},\L')$  with $\L'=\{ X  \backslash e \mid X \in \L\}$, and the \emph{contraction} as an OM  $\M/ e=(U\backslash \{e\},\L'')$ with $\L''=\{ X  \backslash e\ \mid X \in \L, X_e=0\}$, where $X\backslash e$ denotes the element of $\{+,-,0\}^{U\backslash \{e\}}$ with $(X\backslash e)_f=X_f$ for all $f \in U\backslash \{e\}$. It follows from the basic OM theory that both operations in fact result in an OM, see \cite[Section 3.3]{bjvestwhzi-93}. A reverse operation of a deletion is a \emph{single-element extension}, i.e.~$\M$ is a single-element extension of $\M'$ with $e$, if $\M' = \M \backslash e$ for some element $e$.

A \emph{tope} of an OM $\M=(U,\L)$ is defined as a covector with all its coordinates non-zero. We can denote $\T(\M) = \L \cap \{+,-\}^U$ as the set of topes of an OM $\M$. Since the topes are elements of $\{+,- \}^U$, they can  be naturally seen as vertices of the induced subgraph of a hypercube graph $Q_d$, $d=|U|$, whose vertices are elements of $\{+,-\}^U$, and two such vertices adjacent if they differ in exactly one coordinate. The graph obtained in this way is called the \emph{tope} graph of an OM, see Figure \ref{fig:1}; we will denote it with $G(\T(\M))$. It is a basic result that OMs can be reconstructed up to an isomorphism from their tope graphs. With respect to the topological representation, the topes correspond to the  regions that the pseudo-spheres cut $S^{d-1}$ into. For example, in Figure \ref{fig:1} the topes can be seen as 1-dimensional regions of $S^1$ cut-out by the four $S^0$, or equivalently (since the OM is representable), as 2-dimensional regions of $\R^2$ cut-out by the central hyperplane arrangement. In this paper the topes of OMs play a crucial role, since we will be constructing sample compression schemes for classes of topes of OMs.

As with samples, we can impose an ordering $\leq$ on the elements of $\L$ of an OM $\M=(U,\L)$ as the product ordering of elements in $\{+,-,0\}^U$, derived from ordering $0 \leq +$ and  $0\leq -$. The ordering has many nice properties, one of which is that all maximal chains in $(\L, \leq)$ have the same length (number of elements in the chain). This can be used to define the \emph{rank} of an OM, i.e.~$\rank(\M)$, as equal to the length of the maximal chains minus one. As we will see shortly, the rank of an OM corresponds to the VC-dimension of its tope graph. For each covector $X\in \L$ in an OM $\M=(U,\L)$, one can define $\L(X) = \{Y \in \L \mid X \leq Y\}$. Geometrically it can be interpreted as all the covectors of $\M$ corresponding to the regions adjacent to the region of $X$. Going further, we can define $\T(X) = \{Y \in \T(\M) \mid X\leq Y \}$ as the set of all topes in $\L(X)$. It follows from the axioms of OMs, that $\L(X)$ corresponds to an OM, to be more precise $\overline{\L}(X) := \L(X)\backslash \underline{X} = \{Y\backslash \underline{X} \mid X \leq Y\}$ is a set of covectors of an OM, where $\underline{X} = \{e \in U; X_e \neq 0\}$. In fact, due to Axiom (Sym) it could be alternately defined as $\overline{\L}(X) = \L \backslash \underline{X}$, i.e. obtained by deleting all the elements in $\underline{X}$. Then $\T(X)$ induces a tope graph of an OM, that is a subgraph of the tope graph of $\M$. This also allows us to speak of the rank of a covector $X$, meaning the rank of $\overline{\L}(X)$.

The minimal elements of $(\L \setminus \{00\ldots 0\}, \leq)$ are known as \emph{cocircuits}. In particular, for each cocircuit $X$, the set $\T(X)$ induces the tope graph of an OM with the rank one less than the rank of $\M$ (since its maximal chains are exactly one covector shorter).

The first part of the following claim is well known, see \cite[Lemma 13]{chepoi2022ample}. For the sake of completeness we include the proof.
\begin{claim}\label{claim:rankvc}
The rank of a simple OM $\M=(U,\L)$ corresponds to the VC-dimension of its 
tope graph $G(\T(\M))$. If $\rank(\M\backslash e) < \rank(\M)$, for some
element $e\in U$, then there exists cocircuit $X \in \L$, such that the tope
graph $G(\T(\M))$ consists of the induced subgraphs on sets $\T(X), \T(-X)$,
and a matching between them, i.e.~$G(\T(\M))$ is isomorphic to the Cartesian
product of an edge $K_2$ and the graph induced by $\T(X)$.
\end{claim}
\begin{proof}
First, consider the following observation also done in \cite[Proposition 1 in Chapter 7(i)]{Man-82}. If $\rank(\M\backslash e) < \rank(\M)$, for some element $e\in U$, then all the maximal chains in $(\L\backslash e, \leq)$ are shorter than in $(\L, \leq)$. Since we consider simple OMs, there must be a cocircuit $X$ with $X_e \neq 0$ in $\M$. But to shorten a maximal chain that includes the cocircuit $X$ by deleting $e$, it must be that $X\backslash e = 00\ldots 0$, since for all the covectors $Y$, with $X\leq Y$ it holds $Y_e=X_e\neq 0$. This implies that for every tope $T \in \T(\M)$ either holds that $T\geq X$, if $T_e=X_e$, or $T\geq -X$, if $T_e\neq X_e$. In particular, the tope graph $G(\T(\M))$ of $\M$ consists of the induced subgraphs on sets $\T(X), \T(-X)$, and the edges between them. Since $\T(X), \T(-X)$ induce two isomorphic graphs whose topes follow Axiom (Sym), the tope graph $G(\T(\M))$ must be the Cartesian product of the tope graph induced by $\T(X)$ and an edge $K_2$. This proves the second assertion of the claim.

If there exists a cocircuit $X$ with $X_e\neq 0$ and $X\backslash e \neq 00\ldots 0$, then it must be that $\rank(\M\backslash e) = \rank(M)$. Using this argument recursively, one can always delete an element $f$ of $\M$ and have $\rank(\M\backslash f) = \rank(M)$, unless all the cocircuits in $\M$ have exactly one coordinate non-zero. In this case, the tope graph of $\M$ must be isomorphic to $Q_n \cong K_2^n$. Since the OM with the tope graph isomorphic to $Q_n$ has covectors $\L'=\{+,-,0\}^n$, it has rank $n$ (the length of the maximal chains is $n+1$). We have proved the following: if an OM $\M$ has rank $r$, it can be transformed by a sequence of deletions into an OM with the tope graph isomorphic to $Q_r$.

Recall that the $\vc(\C)$ of a class $\C$ is the cardinality of the largest subset of $U$ shattered by $\C$. Translating to the language of $\{+,-\}$ vectors in an OM instead of subsets, it asks to find a subset $S$ of elements $U$ of maximal size, such that deleting all the elements in $U\backslash S$ from an OM $\M$, results in an OM that has the tope graph isomorphic to $Q_{|S|}$. As shown above, this equals the rank of $\M$.
\end{proof}

\vspace{7pt}

\noindent
\textbf{Partial cubes:}
Some results in this paper will extend to a bigger class than the tope graphs of OMs.
\emph{Partial cubes} are graphs that can be isometrically embedded into hypercubes, with respect to the shortest path distance. It is a well-known fact that the tope graphs of OMs are partial cubes, see \cite[Proposition 4.2.3]{bjvestwhzi-93}.
For a partial cube $G$, the isometric embedding into a hypercube is unique up to an isomorphism of the hypercube \cite{Djo-73}, hence we shall consider partial cubes as embedded into $\{0,1\}^U$ and call the coordinates $U$ elements, as in the case of OMs. In fact, all the partial cubes considered in this paper have the embedding naturally defined (for example, inherited from the definition of OMs, in the case of the tope graphs). Since the vertex set of a partial cube $G$, embedded into $\{+,-\}^U$, can be seen as a set system or a concept class, we can define sample compression schemes for it, as well as talk about subsets of $U$ that are shattered by $G$. We will denote with $\overline{X}(G)$ the set of all the subsets of $U$, that are shattered by the vertices of $G$. As in the case of OMs, translating the language of subsets to the language of $\{+,-\}$ vectors, the definition of VC-dimension of a partial cube $G$ embedded into $\{+,-\}^U$  asks to find a subset $S$ of elements $U$ of maximal size, such that deleting all the elements in $U\backslash S$ in $G$, results in a partial cube isomorphic to a hypercube $Q_{|S|}$.

%


For a partial cube $G$, embedded into $\{+,-\}^U$ and an element $e\in U$ we shall denote the two \emph{halfspaces} $H_e^+$ and $H_e^-$ of $G$, defined as all the vertices having the $e$-coordinate $+$ or $-$, respectively. 
A subgraph (or a subset of vertices) $F$ of $G$ is \emph{convex} if for every two vertices $x,y$ of $F$ all the shortest $x,y$-paths lie in $F$.  It is not hard to see that halfspaces of $G$ induce a convex subgraph, and since an intersection of convex subgraphs is convex, also every intersection of halfspaces gives a convex subgraph. In partial cubes the inverse implication is also true: every convex set can be obtained as an intersection of halfspaces \cite{Ban-89}.
For a convex graph $C$ we shall denote  with  $\osc(C)$  the set of the subset of elements of $U$ that \emph{osculate} $C$, i.e.~all the coordinates that the edges connecting $C$ and $G\setminus C$ flip. Similarly, we shall denote with $\cross(C)$  the subset of elements of $U$ that \emph{cross} $C$, i.e.~all the coordinates that the edges in the subgraph  induced by $C$ flip. Since convex sets are the intersections of halfspaces, these two sets are disjoint, and, in fact, a convex set can be represented as the intersection of the halfspaces $C = \bigcap \{H_e^{s_e}\mid e\in \osc(C)\}$, where $s_e$ denotes the sign of the side which includes $C$. We denote with $\H(G)$ the set of all the convex subsets of $G$.





\subsection{Affine OMs, oriented matroid programming, and corners}

\noindent
\textbf{Affine OMs:}
Following \cite[Chapter 10]{bjvestwhzi-93}, we will call an \emph{affine} OM a pair $\A=(\M,g)$, where $\M=(U,\L')$ is an OM and $g\in U$ its element. The set of covectors of $\A$ will be denoted by $\L = \L(\A) = \{X \in \L' \mid X_g = +\}$, i.e. $\A$ can be regarded as a halfspace of $\M$. The topes $\T(\A)$, defined as $\T(\A) = \L \cap \{+,-\}^U$,  correspond to the halfspace $H_g^+$ of the tope graph $G(\T(\M))$. 
Similarly as in OMs, we will refer to the minimal elements of $(\L, \leq)$ as cocircuits. Since the cocircuits of $\A$ are cocircuits of $\M$, they can be used to define $\L(X) = \{Y \in \L \mid X \leq Y\}$. As in the case of OMs, for each cocircuit $X$, deleting the constant elements of $\L(X)$ results in a set of covectors of an OM $\L'(X) := \L(X)\backslash \underline{X} = \{Y\backslash \underline{X} \mid X \leq Y\}$ (also equal to $\L\backslash \underline{X}$ due to Axiom (Sym)), where $\underline{X} = \{e \in U; X_e \neq 0\}$. The rank of $\A$ is again defined as the length of the maximal chains in $(\L,\leq)$. By definition, the rank of $\A$ equals the rank of $\L(X)$, for each of its cocircuits $X$, implying that the rank of $\A$ is one less than the rank of $\M$. The proof of Claim \ref{claim:rankvc} with a minimal modification and also \cite[Lemma 13]{chalopin2022unlabeled} shows, that also the rank of an affine OM $\A$  equals the VC-dimension of its tope graph. Additionally, we define \emph{the plane at infinity} of an affine OM as $\L^{\infty} = \{X\in \L' \mid X_g = 0\}$.

As in the case of OMs, it helps to illustrate these notions using the topological representation, although we shall not use it to derive our results. Since an OM can be represented with pseudo-spheres in $S^{d-1}$ by the Topological Representation Theorem \cite[Theorem 5.2.1]{bjvestwhzi-93}, an affine OM can be seen as a halfspace (half-sphere) of it. Hence each affine OM can be represented as a pseudo-hyperplane arrangement in $\R^{d-1}$, see Figure \ref{fig:2} (for an exact correspondence one would need to be a bit more precise, it suffice to just have an illustration in our case). The cocircuits correspond to the 0-dimensional intersections (points) of pseudo-hyperplanes, topes to the regions of maximal dimension, and $\T(X)$, for a cocircuit $X$, can be seen as the set of all topes whose regions are adjacent to the point corresponding to $X$.

\vspace{7pt}
\noindent
\textbf{OM programming:}
We now turn our attention towards \emph{oriented matroid programming}. Firstly, we define a directed graph  that we will call a \emph{directed cocircuit graph}. Let $\A = (\M,g)$ be an affine OM with rank $d$, for $\M=(U,\L')$ an OM, and $f\in U$ one of its elements, $f\neq g$. The vertices of the  directed cocircuit graph are simply the cocircuits (covectors of rank $d$) of $\A$, while the edges are covectors of rank $d-1$, more precisely two cocircuits $X_1,X_2$ are adjacent if there exists a covector $Y$ such that $X_1< Y, X_2< Y$ and there is no $Z$ with $X_1<Z<Y$ or $X_2<Z<Y$. Some covectors of rank $d-1$ give half-edges in the graph: a cocircuit $X_1\in \L$ has a half-edge if there exists a cocircuit (of $\M$) $X_2\in \L^{\infty}$  and a covector $Y$ such that $X_1< Y, X_2< Y$ (in $(\L', \leq)$) with no $Z$ with $X_1<Z<Y$ or $X_2<Z<Y$ (note that this slightly differs from the standard definition of a cocircuit graph which usually has no half-edges). Finally, we define the orientation of the edges in the above graph with respect to $f\in U$. Let $X_1,X_2$ be adjacent cocircuits. Considering them in $\M$ (i.e. $X_1,X_2\in \L'$), it holds by Axiom (Sym) that also $-X_1\in \L'$. By the Axiom (SE)  applied to $-X_1,X_2$ and the element $g\in U$, there must exist $Z\in \L^{\infty}$ such that $Z_e=(-X_1\circ X_2)_e$  for all $e\in U \setminus \Sep(-X_1,X_2)$. As it turns out (see \cite[Chapter 10]{bjvestwhzi-93} or the below topological intuition) $Z$ is unique and a cocircuit. If $Z_f=+$, orient the edge from $X_1$ to $X_2$, if $Z_f=-$, orient the edge from $X_2$ to $X_1$, and let it be non-oriented if $Z_f=0$. In the case of the half-edges, use $Z=X_2 \in \L^{\infty}$.

The above technical definition has a simpler topological interpretation. On one hand, 0-dimensional intersections of pseudo-hyperplanes, i.e.~points, correspond to cocircuits of $\A$, hence these are the vertices of the cocircuit graph. On the other hand, the 1-dimensional intersections, i.e. pseudo-lines, can be seen as sequences of cocircuits that lie on them. Then two cocircuits $X_1,X_2$ are adjacent if they lie consecutive on a pseudo-line. The edge is oriented from  $X_1$ to $X_2$ if following the line from $X_1$ to $X_2$ leads us to the "point at the infinity" that is on the positive side of the hyperplane $H_f$, and is oriented from  $X_2$ to $X_1$ if it lies on the negative side (equivalently, the line is oriented to cross $H_f$ from the negative to the positive side). Finally, it is not oriented if the line and the hyperplane $H_f$ intersect only in the "point at the infinity", (equivalently, the line does not cross $H_f$). This means that all the edges on a single line are ether undirected, or form a directed path.
For an example, see Figure \ref{fig:2}.



Finally, let  $\A=(\M,g)$ be an affine OM, a halfspace of an OM $\M=(U,\L')$, and $f\in U$ one of its elements, $f\neq g$. For $S\in \{+,-\}^V, V\subseteq U$, define the \emph{polyhedron} $\P(S)$ as $\P(S) = \{X \in \L \mid X_e = \{S_e,0\} \text{ for all } e\in V\}$. Topologically a polyhedron corresponds to the intersection of closed halfspaces. 
The cocircuits of $\A$ that lie in $\P(S)$ induce a subgraph in the directed cocircuit graph of $\A$ with respect to the orientation defined by $f$. We call this graph the \emph{graph of a program}. The task of the \emph{oriented matroid (OM) program} $(\M,g,f,P(S))$ is to find a cocircuit $X$ of $\A=(\M,g)$ that lies in the polyhedron $\P(S)$ and has no in-edges in the graph of the program, oriented with the respect to $f$. We call such $X$ an \emph{optimal} solution to the OM program. See again Figure \ref{fig:2}, for an example.

One of the main results concerning OM programming is the following. For our purposes will say that a polyhedron $P(S)$ of an affine OM $\A$ is \emph{bounded} in the direction of $f$, if there are no in-directed half-edges in the graph of the program of $P(S)$ (this is a much weaker notion of boundedness than the standard one, see \cite[Definition 10.1.1]{bjvestwhzi-93}, but it suffices for our purpose).

\begin{thm}[Theorem 10.1.13 in \cite{bjvestwhzi-93} ]\label{thm:min}
Let $\A=(\M,g)$ be an affine OM, a halfspace of an OM $\M=(U,\L')$, $f\in U, f\neq g$, and $P(S)$ a non-empty polyhedron for some $S\in \{+,-\}^V, V\subseteq U$, bounded in the direction of $f$. Then the OM program $(\M,g,f,P(S))$ has an optimal solution.
\end{thm}

Note that the above theorem is a generalization of the standard linear programming result, which specializes for the case when an OM is represented by (straight) hyperplanes. We refer the reader to \cite[Chapter 10]{bjvestwhzi-93} for all the technical details. In this paper, we only need the existence of a solution for a specific OM program that we construct later in Section \ref{sec:res}. Also note that we define the OM program for any polyhedron, while it is more standard to define it for the polyhedron obtained as the intersection of all positive halfspace in an affine OM. The equivalence by reorientation is trivial. 

\begin{figure}
\centering
\includegraphics[scale=0.35]{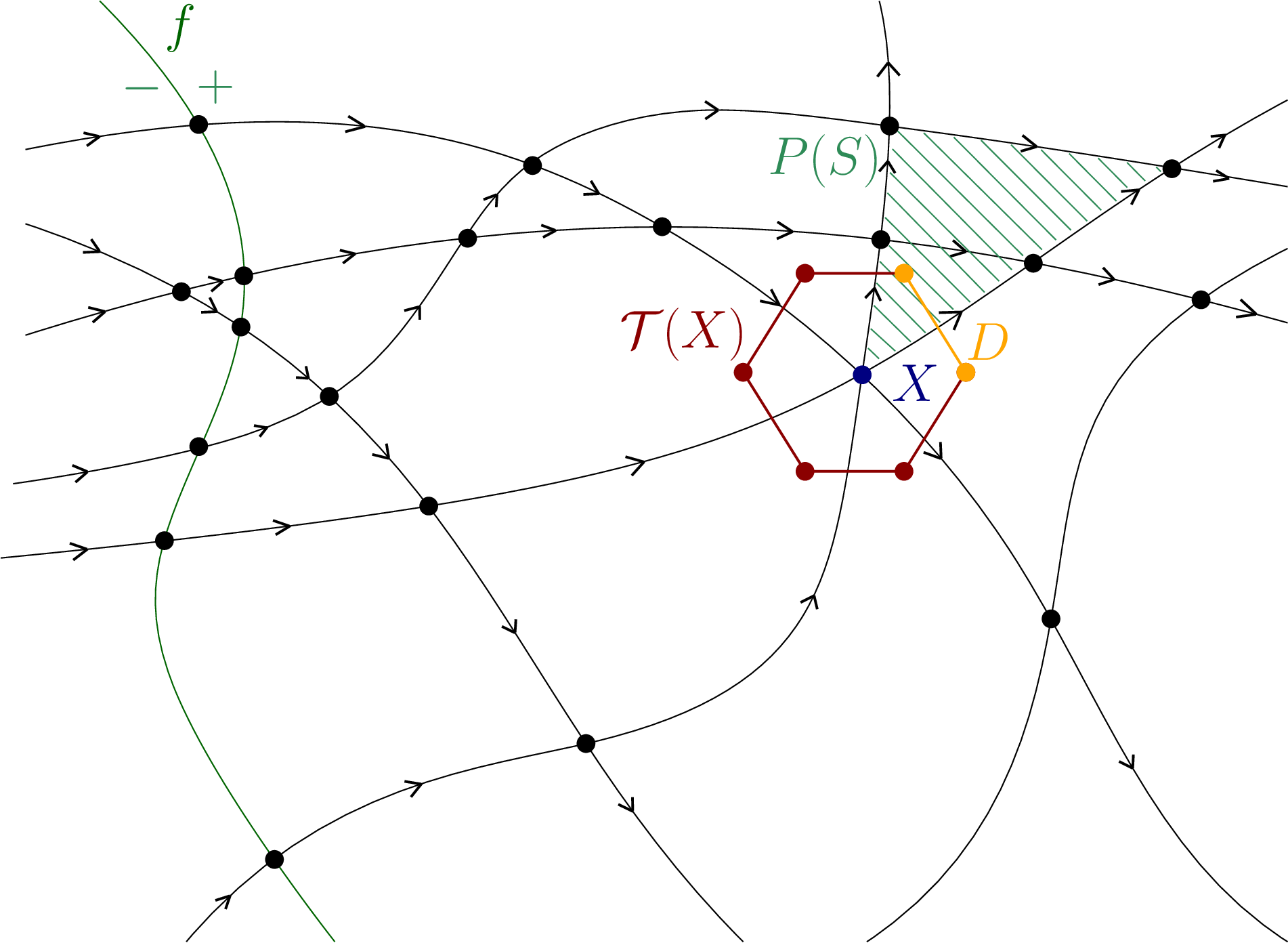}
\caption{The orientation of lines in a rank 2 affine OM $\A$ with respect to a pseudo-hyperplane $f$ resulting in a graph on cocircuits of $\A$. Cocircuit $X$ is an optimal solution to the OM program with respect to polyhedron $P(S)$ (intersection of half-spaces) in direction of $f$. Furthermore, $D$ is a corner of the tope graph induced by $\T(X)$, such that $\T(X) \cap P(S) \subseteq D$, see Lemma \ref{lem:orient_corner}.}
\label{fig:2}
\end{figure}

\vspace{7pt}
\noindent
\textbf{Corners:}
For an element $e\in U$ of  an OM $\M'=(U,\L)$ we will say that it is in \emph{general position}, if none of the $\{ X  \backslash e \mid X \text{ a cocircuit of } \L, X_e=0\}$ is a cocircuit in $\M = \M' \backslash e$. Equivalently, the sets $\{ X  \backslash e \mid X \text{ a cocircuit of } \L, X_e=+\}$ and $\{ X  \backslash e \mid X \text{ a cocircuit of } \L, X_e=-\}$ cover all the cocircuits of $\M' \backslash \{e\}$. Topologically, this can be seen as all the points  on the pseudo-sphere $H_e$,  that correspond to  cocircuits, are no longer corresponding to cocircuits after deleting $H_e$.
Thinking of $\M'$ as a single-element extension of $\M$, we will also say that the extension of $\M$ is in general position, if the corresponding element in the extended OM $\M'$ is in general position.

The last concept that we need to introduce are the \emph{corners} of OMs, first defined in \cite{knauer2020corners}. Let $\M, \M'$ be OMs, where $\M'$ is a single-element extension in general position of $\M$ with element $f\in U$. The two \emph{corners} of $\M$ with respect to the extension with $f$ are the subsets of topes $D = \T(\M) \setminus \{X \backslash f \mid X_f = +, X \in \T(\M')\}$ and $-D=\T(\M) \setminus \{X \backslash f \mid X_f = -, X \in \T(\M')\}$. Note that by the definition $\T(\M)\setminus D = \{X \backslash f \mid X_f = +, X \in \T(\M')\}$ is isomorphic to a halfspace of the tope graph of the OM $\M'$. Hence, by removing a corner from a tope graph of an OM, what remains is the tope graph of an affine OM (isomorphic to the tope graph induced by $\T(\A)$, where $\A=(\M',f)$). We highlight this in the following claim.
\begin{claim}\label{clm:chunk_iso}
Let $D$ be a corner of the tope graph $G$ of an OM. Then $G\backslash D$ is isomorphic to the tope graph of an affine OM.
\end{claim}
Note that $G\backslash D = \{X \backslash f \mid X_f = +, X \in \T(\M')\}$ induces an isometric subgraph of the tope graph of $\M$, see \cite{knauer2020corners} for more details, hence removing a corner does not change the isometric embedding of a graph in a hypercube. 

\vspace{5pt}
\noindent
\textit{Example:} Let $\M$ be an OM with its tope graph isomorphic to a cycle with $2n$ vertices $C_{2n}$ (see the OM in Figure \ref{fig:1}). Every single-element extension in general position with $f$ of this OM results in an OM $\M'$ with its tope graph isomorphic to $C_{2(n+1)}$. Then a corner $D$ of $G(\T(\M))$ consists of $n-1$ consecutive vertices on the cycle. Removing these vertices results in a path of $n+1$ vertices, isomorphic the tope graph of an affine OM $\A=(\M',f)$. Note that such a path is a maximal subgraph of $C_{2n}$ that is isometric. On the other hand, if the tope graph of $\M$ is  isomorphic to $Q_n$, it is not hard to see by analyzing single-element extensions in general position, that a corner of it must be a single vertex. Another example of a corner is given in Figure \ref{fig:3a}.
\vspace{5pt}

%
%

%


In the following, we will establish a connection between corners and solutions to OM programs. We shall use these results later to prove the main theorem. First we need a result from \cite{Man-82}.

\begin{lem}[Theorem 8 and Lemma 11 in Chapter 9(iii) of \cite{Man-82}]\label{lem:mandel_extend}
Let $\A=(\M,g)$ be an affine OM, a halfspace of an OM $\M=(U,\L')$, with $f\in U$ and $X\in \L'$ a cocircutit of $\A$, where $X_f \neq 0$ and $X_g = +$. If $X$ is not in a directed cycle in the cocircuit graph of $\A$ with respect to $f$, then there exists a single-element extension $\A'$ of $\A$ with element $e$, such that:
\begin{itemize}
\item $X$ is extended to a cocircuit $X'$ of $\A'$ with $X'_e=0$.
\item The orientation of the edges in the cocircuit graph of $\A'$ with respect to $e$ is the same as if they were oriented with respect to $f$.
\end{itemize}
  
\end{lem}

A geometric interpretation of the above result is that given a pseudo-hyperplane arrangement representing $\A$, a chosen pseudo-hyperplane $f$ and a cocircit point $X$, one can extend $\A$ to an affine OM with an additional pseudo-hyperplane "parallel" to $f$ and crossing $X$, if the cocircuit graph is acyclic. In the case of affine OMs realizable with (straight) hyperplanes, it is simple to see that this is always possible.

We give certain details about the result in Lemma \ref{lem:mandel_extend} that we will use later, in particular about how the single-element extension of $\A$ to $\A'$ with $e$ is constructed. These details are slightly hidden in \cite{Man-82}, since the existence statement is given in \cite[Theorem 8 in Chapter 9(iii)]{Man-82}, while the construction can be read from the proof of \cite[Theorem 13 in Chapter 9(iii)]{Man-82}. We explain it here in the language of covectors. Let $X$ be a cocircuit of $\A$ and assume that all the edges incident with $X$ in the cocircuit graph of $\A$ are directed. Firstly, $X$ is extended to $X'$ with adding coordinate $e$ where $X'_e=0$. Let $Y\in \L(X)$ be a covector with its rank equal to $\text{rank}(\A) - 1$. Then this covector corresponds to an edge in the directed cocircuit graph with one endpoint $X$. If the edge points towards $X$, the extension of $Y$ is $Y'$ with $Y'_e=-$, otherwise $Y'_e=+$. Furthermore, each covector $Z\in \L(X)$ of rank less than $\text{rank}(\A) - 1$ is extended by the following rule: if each $Y$ with $Y<Z$ and the rank equal to $\text{rank}(\A) - 1$ is extended to $Y'$ with $Y'_e=+$ (resp. $Y'_e=-$), then $Z$ is extended to $Z'$ with $Z'_e = +$ (resp. $Z'_e=-$); otherwise $Z$ is extended to three covectors with the $e$ values $+,-$, and $0$.

The above details will be used in the proof of Lemma \ref{lem:orient_corner}. The statement of the lemma is slightly technical, hence it is illustrated in Figure \ref{fig:2}.

\begin{lem}\label{lem:orient_corner}
Let $\A=(\M,g)$ be an affine OM, a halfspace of an OM $\M=(U,\L)$, and $f\in U$ be in general position in $\M$. Let $X$ be a cocircuit of $\A$, $X_f\neq 0$. Then the OM induced by $\T(X)$ has a corner $D$, such that the following holds: for any polyhedron $P(S)$, $S\in \{+,-\}^V, V\subseteq U$, if $X$ is a solution to the OM program $(\M,g,f,P(S))$, then  $P(S) \cap  \T(X) \subseteq D$. 
%
\end{lem}
\begin{proof}
Firstly, we prove is that if $f$ is in general position, all the edges in the directed cocircuit graph of $\A$ with respect to $f$ are oriented, besides those between cocircuits $X_1,X_2$ with $(X_1)_f=(X_2)_f=0$. Assume on the contrary, that an edge between $X_1,X_2$ is not oriented, hence, by the definition of the orientations, applying (SE) axiom to $-X_1,X_2$ and $g$ yields a unique cocircuit $Z$ with $Z_f=0$. Since $Z_f=0$, it must be that $f\in \Sep(-X_1,X_2)$ or $(X_1)_f=(X_2)_f=0$. Since the latter is one of the cases that we want to prove, assume the former, i.e.~$(X_1)_f=(X_2)_f\neq 0$. Then $X_1 \backslash f, X_2 \backslash f$ are cocircuits in $\M \backslash f$, since the length of the maximal chains including $X_1$, $X_2$ was not shortened by the deletion, and applying (SE) axiom to $-X_1\backslash f,X_2\backslash f$ and $g$ in $\M \backslash f$ gives a unique cocircuit $Z'$. But then $Z' = Z \backslash f$, which is in contradiction with the definition of $f$ being in general position and $Z_f=0$. This proves that the edge must be oriented.

Let $X$ be an arbitrary cocircuit of $\A$ with $X_f\neq 0$ and recall that $\underline{X}  \subseteq U$ denotes the support of $X$. Consider the affine OM $\A' = \A \backslash (\underline{X} \setminus \{g,f\}) := (\M \backslash (\underline{X} \setminus \{g,f\}), g)$ and its directed cocircuit graph with respect to $f$. The image of $X$ under the deletion is  $X'$ with $X'_f=X_f, X'_g=+$, and $X'_e = 0$ for all the other elements $e$ of $U \setminus \underline{X}$. In particular, this implies that $X'$ is the only cocircuit of $\A'$ with $X'_f = X_f, X'_g=+$. Also note that the OM with covectors $\overline{\L}(X)$ is identical to the one with $\overline{\L}(X')$ since only elements with non-zero coordinates in $X$ were deleted.


Now we apply Lemma \ref{lem:mandel_extend} to get a single-element extension of $\A'$ with $e$, such that $X'$ is extended to $X''$ with $X''_e=0$ and the orientation of the lines with respect to $e$ is the same as if they were oriented with respect to $f$. The lemma can be applied, since $X'$ is the only cocircuit of $\A'$ with $X'_f = X_f$, therefore there is no directed cycle in its directed cocircuit graph. We will use this extension to define a corner of $\T(X)$.


Consider the above single-element extension as the extension of $\overline{\L}(X)$. To conclude the proof, we need to prove two facts:
\begin{itemize}
\item The above single-element extension of $\overline{\L}(X)$ with $e$ is in general position, hence defining a corner $D$ of all the topes in $Z \in \T(X)$ that are not expanded to $Z'$ with $Z'_e=-$.
\item The corner $D$ covers $P(S) \cap  \T(X)$. 
\end{itemize}
The rank of the cocircuits of $\overline{\L}(X)$ equals $\text{rank}(\A)-1$, since the rank of $X$ equals $\text{rank}(\A)$. By the construction details of the extension given after Lemma \ref{lem:mandel_extend} and the fact that all the edges incident with $X$ are oriented, each cocircuit $Y$ of $\overline{\L}(X)$ is expanded to $Y'$ with $Y'_e$ equal to $+$ or $-$, depending on whether it corresponds to an out- or an in-edge. Hence the extension is in general position. Finally, every $Y\in P(S) \cap \L(X)$  with rank equal to $\text{rank}(\A) - 1$ (corresponding to the cocircuit of $\overline{\L}(X)$) is expanded to $Y'$ with $Y'_e=+$, since $X$ is a solution of an oriented matroid program with respect to $P(S)$. Hence the same holds for the topes in $P(S) \cap  \T(X)$, by the construction of the expansion, proving that they are in the corner $D$.
%
\end{proof}

\subsection{Reconstructible maps}

We are ready to start analyzing sample compression schemes of OMs. 
In particular, we will give proper unlabeled compression schemes for the concept classes of the topes of OMs. 
We define a map that will help us build sample compression schemes and is closely connected to the \emph{representation} map from \cite{chalopin2022unlabeled}. We define this map for an arbitrary partial cube, since we will later use it for classes beyond OMs. As explained in Section \ref{sec:prelim}, the tope graphs of OMs are partial cubes, with an already defined embedding.

Recall that $\H(G)$ denotes the set of all convex sets in a partial cube $G$ embedded into a hypercube $Q_U$ (equivalently intersections of halfspaces), while $\overline{X}(G)$ are all the subsets of $U$ that $G$ (as a set system) shatters.

\begin{defi}\label{def:reconst}
Let $G$ be a partial cube. We shall say that a map
$$a: \H(G) \rightarrow \overline{X}(G)$$
is a \emph{reconstructible map}, if the following holds:
\begin{enumerate}[(a)]
\item For every convex set $C\in \H(G)$, it holds $a(C) \subseteq \osc(C)$.
\item For every $V \in \Im(a)$, the intersection $\bigcap_{}\{C \mid C\in \H(G), a(C)=V\}$ is non-empty.
\end{enumerate}
\end{defi}

The following connects reconstructible maps with  proper unlabeled compression schemes of small size. Note that the connection between samples and convex sets was also established in \cite{chepoi2021labeled}.

\begin{lem}\label{lem:reconstruct}
If a partial cube $G$ has a reconstructible map, then $G$ (as a set system) has a proper unlabeled compression scheme of size at most $\vc(G)$. 
\end{lem}
\begin{proof}
Assume $G$, embedded into $\{+,-\}^U$, has a reconstructible map $a$.
Let $s \in \RS(G)$ be a sample realizable by $G$ and $v$ a vertex of $G$ that realizes $s$. Define a convex set $C = \bigcap \{H_e^{s_e} \mid e \in U, s_e \in \{+,-\}\}$, i.e. $C$ is the intersection of halfspaces of $G$ defined by the coordinates of $s$. By definition, $v \in C$, hence $C$ is non-empty. Define $\alpha(s) = a(C)$. Moreover, for any $V\subseteq \Im(\alpha)$ define $\beta(V)$ to be any vertex from the intersection $\bigcap \{C \mid C\in \H(G), a(C)=V\}$, which is non-empty by the definition of the reconstructible maps.

We now prove that $\alpha, \beta$ form a proper unlabeled compression scheme of size at most $\vc(G)$. First, we argue that $\alpha(s) \subseteq \underline{s}$ for every realizable sample $s$. Let $C$ be the convex set corresponding to $s$, as defined above. By the definition of $C$, for every $e\in \osc(C)$, it must hold that $s_e \in \{+,-\}$.  Hence $\osc(C) \subseteq \underline{s}$, and since $\alpha(s) = a(C) \subseteq \osc(C)$, the condition holds.

Secondly, we argue that $s \leq \beta(\alpha(s))$ for every $s \in \RS(G)$. Let $C$ be the convex set corresponding to $s$ and let $s_e  \in \{+,-\}$ for some $e\in U$. Then by definition,  $C$ is a subset of $H_e^{s_e}$, i.e. $c_e = s_e$ for every vertex $c$ of $C$. In particular this holds for $c = \beta(\alpha(s))$, since $c$ is chosen from the intersection $\bigcap \{C \mid C\in \H(G), a(C)=V\}$. We have proven that $s_e = \beta(\alpha(s))_e$ if $s_e  \in \{+,-\}$, while if $s_e=0$, it clearly holds $s_e \leq \beta(\alpha(s))_e$. 

Finally, the size of the compression scheme is bounded by $\vc(G)$, since, by the definition of VC-dimension, the sets in $\overline{X}(G)$ are bounded by $\vc(G)$. The compression scheme is proper, since $ \beta(\alpha(s)) = c $ is a vertex of $G$.
\end{proof}

In \cite{chepoi2021labeled}, a map close to a reconstructible map was proven to exist for OMs, which was used by the authors to establish labeled compression schemes. The map does not map all convex sets but only some, that we define now. Let $\M$ be an OM, and $C\subseteq \T(X)$ a subset of topes of $\M$ that is convex in the tope graph $G(\T(\M))$. Then $C$ is said to be \emph{full}, if $\M \backslash \cross(C)$ has the same rank as $\M$. We denote with $\H_f(G)$ the set of all full convex subsets of $G=G(\T(\M))$. We restate their result in our setting:
 
\begin{lem}[Lemma 20 in \cite{chepoi2021labeled}]\label{lem:fullconv}
Let $G=G(\T(\M))$ be the tope graph of an OM $\M$. Then there exists a map 
$$b: \H_f(G) \rightarrow \overline{X}(G)$$
such that 
\begin{enumerate}[(a)]
\item For every convex set $C\in \H_f(G)$, it holds $|b(C)| = \vc(G)$.
\item For every convex set $C\in \H_f(G)$, it holds $b(C) \subseteq \osc(C)$.
\item Let $V = b(C')$ for some $C'\in \H_f(G)$ and let $s \in \{+,-\}^V$ be such that $C' \subset H_e^{s_e}$ for each $e \in V$. Then the intersection $\bigcap \big\{C\in \H_f(G) \mid b(C)=V \text{ and } C \subset H_e^{s_e}, \forall e \in V \big\}$ is non-empty.
\end{enumerate}
\end{lem}

The map fails to be a reconstructible map in two ways. Firstly, it maps only full convex sets. Additionally, it does not guarantee that the intersections of all the convex sets that map to the same set are non-empty, but only for those that lie in appropriate halfspaces. 
We will be using this map for mapping convex sets in a corner. Hence we need to prove the following.
\begin{lem}\label{lem:cornerfull}
Let $G=G(\T(\M))$ be the tope graph of an OM $\M$ and $D\subseteq \T(\M)$ a corner of $G$. Then every convex set $C$ in $G$, that is a subset of $D$, is full. Moreover, the
intersection $\bigcap \big\{C\in \H(G) \mid C\subset D, b(C)=V \big\}$ is non-empty for every $V = b(C')$ with $C'\subset D$.
\end{lem}
\begin{proof}
%
Let $\M'$ be a single-element extension of $\M$ with element $f$ in general position, such that $D$ is a corner of $\M$, say $D = \T(\M) \setminus \{X \backslash f \mid X_f = -, X \in \T(\M')\}$. First we establish the following: The OM $\mathcal{M}'\backslash e$, for every $e\in U$, is a single-element extension of $\M\backslash e$. We argue that, if $\rank(\M \backslash e) = \rank(\M)$, then the extension is in general position.  If it is not in general position, there exists a cocircuit $X$ of $\M'\backslash e$ with $X_f=0$ such that $X \backslash f$ is a cocircuit of $\M \backslash e$. But then the pre-image (with respect to the deletion of $e$) of $X$, say $X'$ in $\M'$, is such that $X'_f=0$ and $X'\backslash f$ (covector of $\M$) is the preimage of $X \backslash f$, hence a cocircuit. This cannot be, since $\M'$ is an extension in general position of $\M$.

Let $C$ be a convex set with $C\subset D$. We now prove that  $C$ is full, i.e.~$\rank(\M \backslash \cross(C)) = \rank(\M)$. Assume on the contrary that this is not the case. Then it holds for some $e_1,\ldots, e_i,e_{i+1} \in \cross(C)$, that $\rank(\M \backslash \{e_1,\ldots, e_i, e_{i+1} \})< \rank(\M \backslash \{e_1,\ldots, e_i\}) = \rank(\M)$, $0\leq i < |\cross(C)|$. By the second assertion of Claim \ref{claim:rankvc}, if $\rank(\M \backslash \{e_1,\ldots, e_i, e_{i+1} \})< \rank(\M \backslash \{e_1,\ldots, e_i\})$, the tope graph $G$ of $\M \backslash \{e_1,\ldots, e_i, e_{i+1} \}$ 
is the union of the subgraphs induced by $\T(X)$ and $\T(-X)$ and the edges between them, for some cocircuit $X$. This implies that every tope of $\M \backslash \{e_1,\ldots, e_i, e_{i+1} \}$ has a pre-image in $\T(X)$ and in $\T(-X)$. 

By the arguments from above, $\M' \backslash \{e_1,\ldots, e_i\}$ is a single-element extension in general position with $f$ of $\M \backslash \{e_1,\ldots, e_i\}$. This implies that there exists a cocircuit $X'$, with $X'_f=-$, in $\M' \backslash \{e_1,\ldots, e_i\}$ that is mapped to $X$ or $-X$, under the deletion of $f$. If $X''$ is a cocircuit of $\M'$ that maps to $X'$ under the deletion of $\{e_1,\ldots, e_i\}$, then $X''_f=-$ and $X''$ is mapped to $X$ or $-X$ under the deletion of $\{f,e_1,\ldots, e_i\}$. Then the image of $\T(X'')$ with respect to the deletion of $\{e_1,\ldots, e_i, e_{i+1}, f\}$ covers all the topes of $\M \backslash \{e_1,\ldots, e_i, e_{i+1} \}$, hence also all the topes of $\M \backslash \cross(C)$.

Now consider the image of $C$ under the deletion of $\cross(C)$, say $C'$. Since none of the elements that osculate $C$ was deleted, the topes in the preimage of $C'$, under the deletion of $\cross(C)$, are exactly the topes in $C$. On the other hand, we have proved above that there is a cocircuit $X''$ in $\M'$, with $X''_f=-$, such that $\T(X'')$ covers all the topes of $\M \backslash \cross(C)$ under the deletion, hence also $C'$. But $C\subseteq D = \T(\M) \setminus \{X \backslash f \mid X_f = -, X \in \T(\M')\}$, which is a contradiction. Thus  $\M \backslash \cross(C)$ must have the same rank as $\M$. This proves the first statement of the lemma.

%


%
%
%
%
%
%



Now let there be $C^1, C^2 \subseteq D$ such that $b(C^1) = b(C^2)$. Since $b(C^1)$ is shattered by $G$ and $|b(C^1)| = \vc(G)=\rank(\M)$, by Claim \ref{claim:rankvc}, the rank of $\M \backslash (U \setminus b(C^1))$ equals the rank of $\M$ and the tope graph of $\M \backslash (U \setminus b(C^1))$ is isomorphic to a hypercube of dimension $\vc(G)$. As proved above, then $\M' \backslash (U \setminus b(C^1))$ is a single-element extension in general position of $\M \backslash (U \setminus b(C^1))$, thus the image of the corner $D$ under the deletion must be a corner of  $\M \backslash (U \setminus b(C^1))$. Since the corners of an OM with tope graph isomorphic to a hypercube correspond to single vertices, it must be that the image of $D$ is a vertex. This implies that the image of $C^1\subseteq D$ as well as $C^2\subseteq D$ under the deletion of $U \setminus b(C^1)$ is also the beforementioned vertex, hence they are equal. In particular, for every $e\in b(C^1)$ the sets $C^1$ and $C^2$ lie in the same halfspace with respect to $e$. This means that for $s \in \{+,-\}^V$, such that $C^1 \subset H_e^{s_e}$ for each $e \in b(C^1)$, it holds $\bigcap \big\{C\in \H(G) \mid C\subset D, b(C)=b(C^1) \big\} = \bigcap \big\{C\in \H(G) \mid C\subset D, b(C)=b(C^1), C \subset H_e^{s_e}, \forall e \in b(C^1) \big\} \supset \bigcap \big\{C\in \H(G) \mid b(C)=b(C^1), C \subset H_e^{s_e}, \forall e \in b(C^1) \big\}$. By the property (c) of $b$ in Lemma \ref{lem:fullconv}, the latter is non-empty.
\end{proof}

\section{Main result}\label{sec:res}

The main result of this paper is the following.

\begin{thm}\label{thm:om_rep}
For every OM $\M$, its tope graph has a reconstructible map, thus the concept class of topes $\T(\M)$ has
a proper unlabeled compression scheme of size $\vc(\T(\M))$.
\end{thm}
The theorem follows from Lemma \ref{lem:reconstruct} and Propositions \ref{prop:extend} and \ref{prop:afine}, given bellow, in the following way. Recall from Claim \ref{clm:chunk_iso}, that removing a corner of the tope graph $G$ of an OM $\M$ results in a graph isomorphic to the tope graph of an affine OM. Moreover, every OM has a corner, hence this operation can always be done. The resulting graph is guarantied to have a reconstructible map by Proposition \ref{prop:afine}. Then Proposition \ref{prop:extend} allows us to extend this map to a reconstructible map of $G$. This process is illustrated in Figure \ref{fig:3}. Finally, Lemma \ref{lem:reconstruct} gives a proper unlabeled compression scheme of size at most $\vc(G)$.

\begin{figure}[ht]
\centering
\begin{subfloat}[A corner of the tope graph of an OM.\label{fig:3a}]
{\includegraphics[scale=0.23]{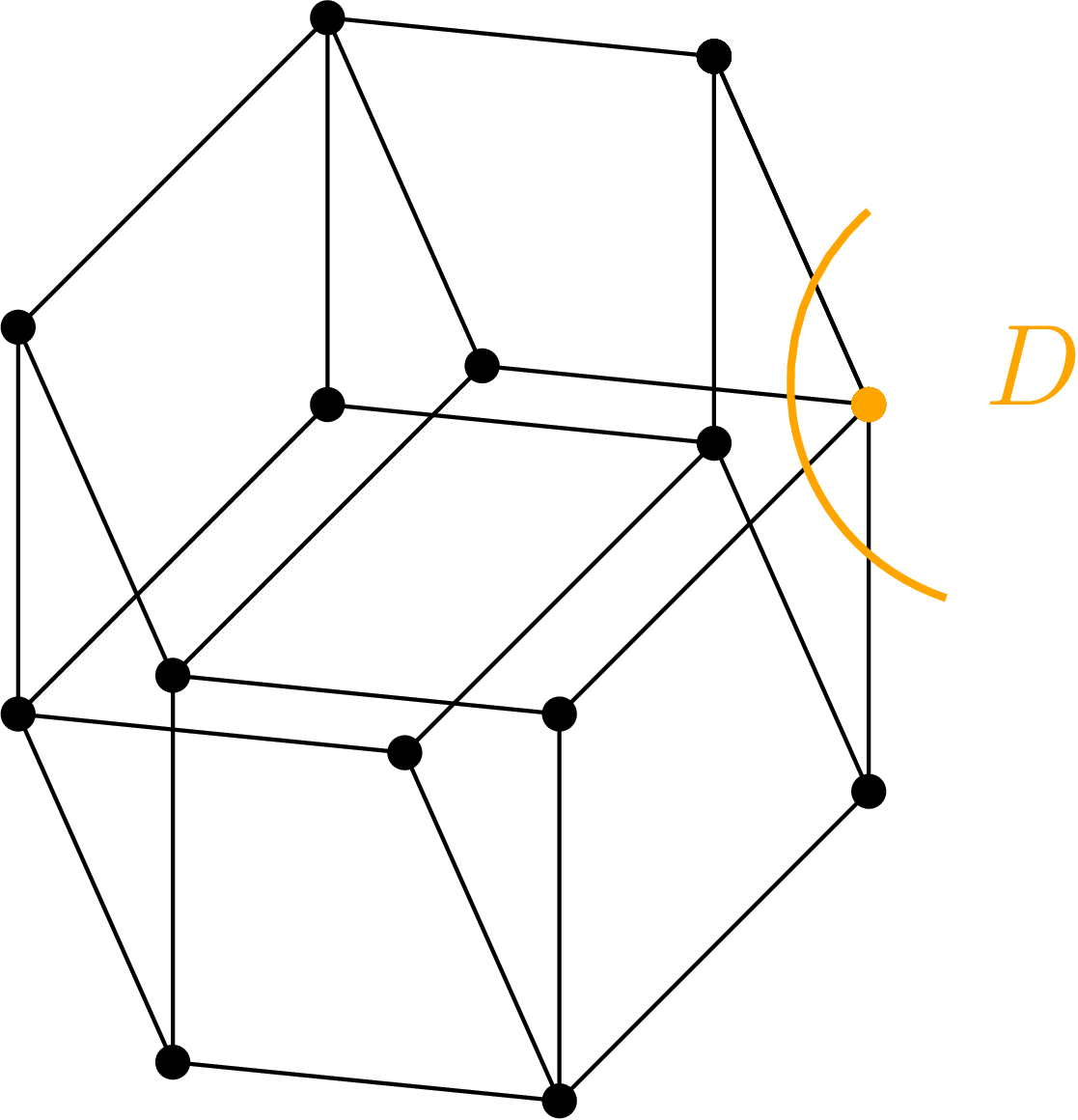}}
\end{subfloat}
\hspace{40pt}
\begin{subfloat}[To define a reconstructible map we add a hyperplane $f$ and for each convex set $C$ solve an OM program to obtain a solution $X$. The reconstructible map of $\T(X)$ helps us define the image of $C$.\label{fig:3b}]
{\includegraphics[scale=0.14]{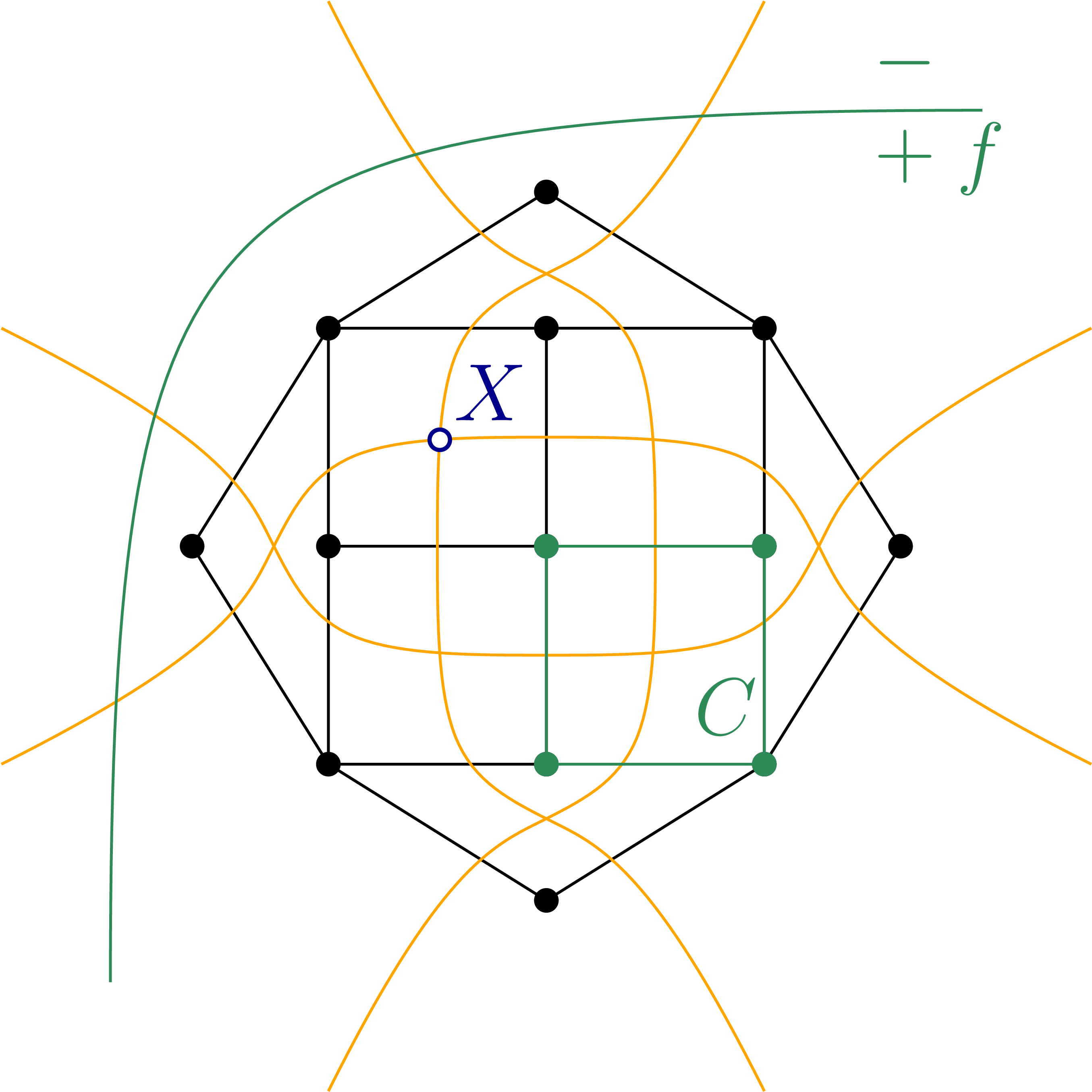}}
\end{subfloat}
\caption{Defining a reconstructible map for the tope graph of an OM: cut a corner of the tope graph of an OM as in (a), define a reconstructible map for the remaining tope graph of an affine OM by Proposition \ref{prop:afine} as in (b), and extend the map to the OM by Proposition \ref{prop:extend}.}
\label{fig:3}
\end{figure}

We now give details about the second part of the above process. If $G_1$ is an isometric subgraph of a partial cube $G$, then for every set $C$ that is convex in $G$, the intersection $C \cap G_1$ is convex in $G_1$, assuming the intersection is non-empty. If an element osculates $C \cap G_1$ in $G_1$, it osculates $C$ in $G$.  This allows us to say that a reconstructible map $a$ of $G$ \emph{extends} a reconstructible map $\overline{a}$ of $G_1$ if $a(C) = \overline{a}(C \cap G_1)$ for every convex set $C$ of $G$ with $C \cap G_1 \neq \emptyset$ (to be strictly precise, not every convex set of $G_1$ is convex in $G$, hence not all the information is used in the extension).

\begin{prop}\label{prop:extend}
Let $G=G(\T(\M))$ be the tope graph of an OM $\M$ and let $D$ be a corner of $G$. Assume that the graph $G \setminus D$ has a reconstructible map $\overline{a}$. Then $\overline{a}$ can be extended to a reconstructible map $a$ of $G$, with $|a(C)| = \vc(G)$ for every $C\subseteq D$.
\end{prop}
\begin{proof}
Denote $G_1 = G \setminus D$. By the definition of a corner, there exists a single-element extension in general position $\M'$ of $\M$ with element $f$, such that $G_1$ is the subgraph of $G$ induced by $\{X \backslash f \mid X_f = +, X \in \T(\M')\}$. Since $\{X  \mid X_f = +, X \in \T(\M')\}$ are precisely the topes of the affine OM $\A=(\M,f)$, $G_1$ is in fact (isomorphic to) the tope graph of an affine OM, as stated in Claim \ref{clm:chunk_iso}. Moreover, the embedding of $G_1$ in $G$ is isometric and provided by the definition, hence we can speak of extending a reconstructible map of $G_1$ to $G$.

Due to Claim \ref{claim:rankvc}, $\vc(G_1) = \vc(G)-1$.
Let $\overline{a}: \H(G_1) \rightarrow \overline{X}(G_1)$ be the assumed reconstructible map of $G_1$. 
Define $a: \H(G) \rightarrow \overline{X}(G)$:
\begin{enumerate}[(i)]
\item If $C\in \H(G)$ is such that $C \cap G_1 \neq \emptyset$, then define $a(C) = \overline{a}(C \cap G_1)$.
\item If $C\in \H(G)$ is such that $C \subseteq D$, then define $a(C) = b(C)$, where $b$ is the map from Lemma \ref{lem:fullconv}. By Lemma \ref{lem:cornerfull}, every convex $C \subseteq D$ is full, hence the map is well defined.
\end{enumerate}

Each $V\in \Im(a)$ is in case (i), by the definition of $\overline{a}$, shattered by $G_1$ (hence also by $G$), and shattered by $G$ in case (ii), by the definition of $b$.
Now we prove that $a$ is a reconstructible map, by proving that conditions (a) and (b) from Definition \ref{def:reconst} hold. Note that the convex sets from the first condition map to sets of cardinality strictly less than $\vc(G)$, while the convex sets in the second condition map to sets of cardinality exactly $\vc(G)$.
\begin{enumerate}[(a)]
\item For every convex set $C\in \H(G)$, it holds $a(C) \subseteq \osc(C)$:

In the case that $C$ is such that $C \cap G_1 \neq \emptyset$, the condition holds, since if $e\in \osc_{G_1}(C \cap G_1)$, then $e\in \osc_G(C)$, as explained before.

In the case that $C \subseteq D$, the condition follows by the property (b) of the map $b$.
\item For every $U \in \Im(a)$, the intersection $\bigcap\{C\mid C\in \H(G), a(C)=U\}$ is non-empty.

If $|U| < r$, then  $\bigcap\{C \mid C\in \H(G), a(C)=U\} \supseteq \bigcap\{C \mid C \in \H(G_1), \overline{a}(C)=U\} \neq \emptyset$,
since $U$ is in the image of $\overline{a}$ and $\overline{a}$ is a reconstructible map.

In the case $|U| = r$, $\bigcap\{C \mid C\in \H(G), a(C)=U\} = \bigcap\{C \mid C\in \H_f(G), C\subset D, a(C)=U\}$, which is non-empty by Lemma \ref{lem:cornerfull}. Hence the property follows.
\end{enumerate}
\end{proof}

\begin{prop}\label{prop:afine}
The tope graph $H=G(\T(\A))$ of every affine OM $\A$ has a reconstructible map, thus the class of topes $\T(\A)$ has a proper unlabeled sample compression scheme of size bounded by $\vc(\T(\A))$.

\end{prop}
\begin{proof}
Let $H$ be the tope graph of an affine OM $\A=(\M, g)$, where $\M=(U,\L)$ is an OM. We start the proof by a single-element extension of $\M$ (hence also $\A$). The so-called perturbations, see \cite[Definition 7.2.3]{bjvestwhzi-93}, allow us to extend $\M$, with an element $f$ in general position, to an OM $\M'$, so that $f$ can be seen as a slightly perturbed $g$. More precisely, the extension can be defined such that $\{X \backslash f \mid X \in \T(\M'), X_f=+\} \supseteq \{X \in \T(\M) \mid X_g=+\}=\T(\A)$, i.e.~the image of the halfspace of $\M'$ with respect to $f$ under the deletion of $f$ covers the topes of $\A$. We will now consider the affine OM $\A'=(\M',g)$, see Figure \ref{fig:3b} for an example.

Firstly, we define $H_1$ as the subgraph of $H$ induced on $\{X \backslash f \mid X \in \T(\M'), X_f=-, X_g=+\} \subseteq \T(\A')$. Since the single-element extension $\M'$ with $f$ was defined so that $\{X \backslash f \mid X \in \T(\M'), X_f=+\}$ covers $H$, every tope $X$ in $H_1$ is an image under the deletion of $X_1\in \T(M')$ with $(X_1)_f = +$ as well as of  $X_2\in \T(M')$ with $(X_2)_f = -$. Applying the (SE) axiom to $X_1$ and $X_2$ with respect to $f$ one obtains that there exists $Y$ a covector of $M'$, with $Y_f=0$, and $Y_e\neq 0$ for $e\neq f$, that maps to $X$ under the deletion of $f$.
What we have proved is that the set $H_1'= \{X  \mid X \in \L(\M'), X_f=0, X_g=+, X_e\neq 0 \text{ for } e\neq f\}$ mapped by the deletion of $f$ covers $H_1$. Moreover, applying (C) axiom in $\M'$ to elements of $H_1$ and any tope $Z$ with $Z_f=-$, one sees that, in fact, there is a one-to-one correspondence between elements of $H_1'$ and  $H_1$.

The above implies that $H_1$ is precisely the tope graph of an affine OM $\A''=(\M'/f,g)$, where $\M'/f$ is the contraction of OM, defined in Section \ref{sec:prelim}. See also \cite[Section 3.3]{bjvestwhzi-93} for more details on this operation. In fact, $\rank{\A''} = \rank(\A) - 1$, hence  we can inductively assume that $H_1$ has a reconstructible map $\overline{a}$. Since $H_1$ is an isometric subgraph of $H$ (with the embedding into a hypercube inherited from $H$), we can deduce that the intersections of the convex sets of $H$ with $H_1$ are convex in $H_1$, and we can extend $\overline{a}$ to a reconstructible map $a$ of $H$.

Define $a$ in the following way:
\begin{enumerate}[(i)]
\item If $C \subseteq H$ is such a convex set, that $C \cap H_1 \neq \emptyset$, then define $a(C) = \overline{a}(C \cap H_1)$.
\item If $C \subseteq H$ is such a convex set, that $C \cap H_1  = \emptyset$, then define $a(C)$ in the following way. First let $S\in \{+,-\}^{\osc(C)}$ be such, that $S_e=+$ if $C \subseteq H_e^+$, or  $S_e=-$ if $C \subseteq H_e^-$. Then $P(S)$ is as a polyhedron in $\M'$ (note, we consider the poyhedron in the extension of $\M$). The topes in $C'=P(S) \cap \T(\M')$ are precisely those, that map to $C$ under the deletion of $f$. Since $C \cap H_1  = \emptyset$, it holds that $Y_f=+$ for every $Y \in P(S)\cap \T(\M')$, and the map between vertices of $C'$ and $C$ is bijective.
Let a cocircuit $X$ be a solution to the OM program $(\M',g,f,P(S))$, given by Theorem \ref{thm:min}. In fact, since $C \cap H_1  = \emptyset$ and hence $P(S) \subseteq H_f^+$, the program is bounded, and furthermore $X_f=+$. By Lemma \ref{lem:orient_corner}, $P(S) \cap \T(X)$ is included in a corner $D$ of the tope graph induced by $\T(X)$. Inductively, we can assume that the tope graph induced by $\T(X) \setminus D$, which is isomorphic to the tope graph of an affine OM, has a reconstructible map. The map can be extended to a reconstructible map $b_X$ of the graph induced by $\T(X)$, by Proposition \ref{prop:extend}. In fact, $b_X$ maps convex subsets of $D$ into sets of order $\vc(\T(X))= \rank(\overline{\L}(X))=\rank(\A)=\vc(H)$, by Proposition \ref{prop:extend} and Claim \ref{claim:rankvc}. Finally, define $a(C) = b_X(C' \cap \T(X))$. Note that $C' \cap \T(X)$ is convex in the graph induced by $\T(X)$, since $C'$ and $T(X)$ are convex sets in the tope graph of $\A'$. Moreover, since $X_f=+$, all the topes in $\T(X)$ are in $H_f^+$, hence none of the convex subsets of $\T(X)$ are osculated by $f$ (in the tope graph induced by $\T(X)$). Then $f$ is not an element of the images of $b_X$, and $a$ is well defined. See Figure \ref{fig:3b} for an example.
\end{enumerate}

The VC-dimension of an affine OM is the same as the VC-dimension of the tope graphs of its cocircuits, by Claim \ref{claim:rankvc}. Hence, note that in the case (i) $|a(C)| < \vc(H)$, since $\overline{a}$ is a reconstructible map of $H_1$ with $\vc(H_1) = \vc(H) - 1$. In case (ii) $|a(C)| = \vc(H)$, as explained above, hence there are no collisions between the images in the two cases.

We now prove that $a$ is a reconstructible map. 
\begin{enumerate}[(a)]
\item For every convex set $C\in \H(H_1)$ it holds $a(C) \subseteq \osc(C)$:

First we analyze case (i), hence we assume that $C$ is such that $C \cap H_1 \neq \emptyset$. Every $e$ that osculates  $C \cap H_1$ in $H_1$ also osculates $C$ in $H$, since $H_1$ is an isometric subgraph of $H$ (having the embedding into a hypercube inherited from $H$). This implies that $a(C) = \overline{a}(C \cap H_1) \subseteq \osc_{H_1}(C \cap H_1) \subseteq \osc_H(C)$.

A similar analysis can be done in case (ii). Assume $e$ osculates  $C' \cap \T(X)$ in $\T(X)$ where $C'$ is the convex set consisting of all the topes in $\M'$ that map to $C$ and is used to define $a(C) = b_X(C' \cap T(X))$. Let us denote with $\T(X) \backslash f$ the image of $\T(X)$ under the deletion of $f$ in $\M'$, i.e.~$\T(X) \backslash f$ is a subset of $H$. Since all the topes in $\T(X)$ as well as in $C'$ are in $H_f^+$, also the image of $C'\cap \T(X)$ under the deletion of $f$ is osculated by $e$ in the subgraph induced by $\T(X) \backslash f$. We can write $(C'\cap \T(X))\backslash f= C\cap (\T(X)\backslash f)$ in $\T(X)\backslash f$. Since $\T(X)\backslash f$ induces an isometric subgraph of $H$, $e$ osculates $C$ in $H$. Hence $a(C) = b_X(C' \cap T(X)) \subseteq \osc_{\T(X)}(C' \cap \T(X)) \subseteq \osc_H(C)$.

\item For every $V \in \Im(a)$, the intersection $\bigcap\big\{C\mid C\in \H(H), a(C)=V\big\}$ is non-empty:

If $|V| < \vc(H)$, then all the convex sets that are mapped into $V$ are from the case (i), hence $\bigcap\big\{C \mid C\in \H(H), a(C)=V\big\} \supseteq \bigcap\big\{C'\mid C\in \H(H_1), \overline{a}(C')=V\big\} \neq \emptyset$, since $\overline{a}$ is a reconstructible map.

Let now $|V| = \vc(H)$. First, we claim that there is a unique cocircuit  $X$ of $\A'$, such that $\T(X)$ shatters $V$. This could be easily seen from the topological representation of AOMs and the fact that the rank of an OM matches the VC-dimension of its tope graph. Nevertheless, we prove it for the sake of completeness. Since the topes of $\A$ shatter $V$, also the topes of $\A'$ shatter it.
 Seeing the latter as a set of subsets of $U\cup \{f\}$ instead of $\{+,-,0\}^{U\cup \{f\}}$ vectors, $V$ being shattered by $\T(A')$ translates into $\A' \backslash ((U\cup \{f\})\backslash V)$ having the tope graph isomorphic to the hypercube $Q_{|V|}$. Since then $\A' \backslash ((U\cup \{f\})\backslash V)$ has the set of covectors $\L'=\{+,-,0\}^V$, there must be a cocircuit $X$ of $\A'$ that maps into $00\ldots0 \in \L'$ and the topes $\T(\A')$ must map to elements $\{+,-\}^V$ under the deletion of $(U\cup \{f\})\backslash V$. Then $X\circ Y\in \T(X)$ for every $Y\in \T(\A')$ by the Axiom (C). The image of the latter topes under the deletion covers $\{+,-\}^V$, hence $\T(X)$ shatters $V$. If there are two cocircuits $X_1,X_2$ with this property, then by a similar analysis $\T(A') \supseteq \T(X_1) \cup \T(X_2)$ would shatter $V\cup \{e\}$ for $e\in \Sep(X_1,X_2)$, which cannot be, since $\vc(\T(A'))=\rank(A')=\rank(A)=\vc(\T(A'))=|V|$.

We have proved that $V$ can be mapped to a unique cocircuit $X$ of $\A'$, hence we know that $V\in \Im(a)$ was obtained through an OM program whose solution was the cocircuit $X$. Furthermore, by Lemma \ref{lem:orient_corner}, independent of the polyhedron used in the OM program (as long as $X$ is a solution), there is a corner $D$ of the OM with topes $\T(X)$, such that $V\in \Im(b_X)$, where $b_X$ was defined using the corner $D$.
Hence, $\bigcap\big\{C \mid C\in \H(H), a(C)=V\big\}$ includes the image of $\bigcap\big\{C' \mid C'\in \H(\T(X)),~X \text{ is a solution of an OM program } (\M',g,f,P(S)), P(S)\cap \T(X)=C', b_X(C')=V\big\}$
under the deletion of $f$. Furthermore, the latter includes $\bigcap\big\{C' \mid C'\in \H(\T(X)), b_X(C')=V\big\}$.
Since $b_X$ is a reconstructible map (fixed for all the polyhedrons with solution $X$), the last intersection is non-empty. This proves that the first intersection is non-empty as well.

\end{enumerate}
Finally, every $V\in \Im(a)$ is shattered by $H$, since it is shattered by $H_1$ or by $\T(X)$ for some cocircuit $X$ of $H$.

This finishes the proof.
\end{proof}

\subsection*{Proper unlabeled sample compression schemes and corner peeling}

\emph{Complexes of Oriented Matroids} (COMs) were introduced in \cite{Ban-18} as a generalization of OMs (and affine OMs) preserving many properties without demanding to be centrally symmetric. The class of tope graphs of COMs includes many interesting classes of graphs, such as ample classes (lopsided sets) \cite{bachdrko-06}, hypercellular graphs \cite{Che-16}, median graphs \cite[Chapter 12]{ham-11}, etc.  They can be formally defined by replacing the Axiom (Sym) with a weaker one.

\begin{defi}
A \emph{complex of an oriented matroids (COM)}  is a system of sign vectors $\M=(U,\L)$ satisfying properties \textbf{(C)}, \textbf{(SE)}, and
\begin{itemize}
\item[\textbf{(FS)}] $X\circ -Y \in  \covectors$  for all $X,Y \in  \covectors$.
\end{itemize}
\end{defi}

It remains an open problem to find unlabeled sample compression schemes for ample classes, hence it is also unknown if unlabeled sample compression schemes for the tope graphs of COMs exist. One way to analyze COMs is to see them as a union of OMs glued together in a particular way, see \cite{Ban-18}. The reason for this is that, similarly as in the case of OMs, one can define $\L(X) = \{Y \in \L \mid X \leq Y\}$  and $\overline{\L}(X) := \L(X)\backslash \underline{X}$ for each covector $X\in \L$ and prove that $\overline{\L}(X)$ is an OM. Hence the set of covectors $\L$ of a COM $\M$ is a union $\L = \L(X_1) \cup \cdots \cup L(X_n)$, where $X_1,\ldots, X_n$ are minimal covectors of $\M$. We will call each $L(X_i)$ a \emph{sub-OM}. 

In \cite{knauer2020corners}, a \emph{corner} of the tope graph of a COM $G$ was defined as such a set $D\subseteq G$ with all the vertices of $D$ lying in the tope graph of a unique maximal sub-OM, of which $D$ is a corner. Equivalently, all the vertices $D$ lie in exactly one maximal convex subgraph $H$, that is the tope graph of an OM, and $D$ is a corner of $H$, see Figure \ref{fig:4}.

\begin{figure}
\centering
\includegraphics[scale=0.18]{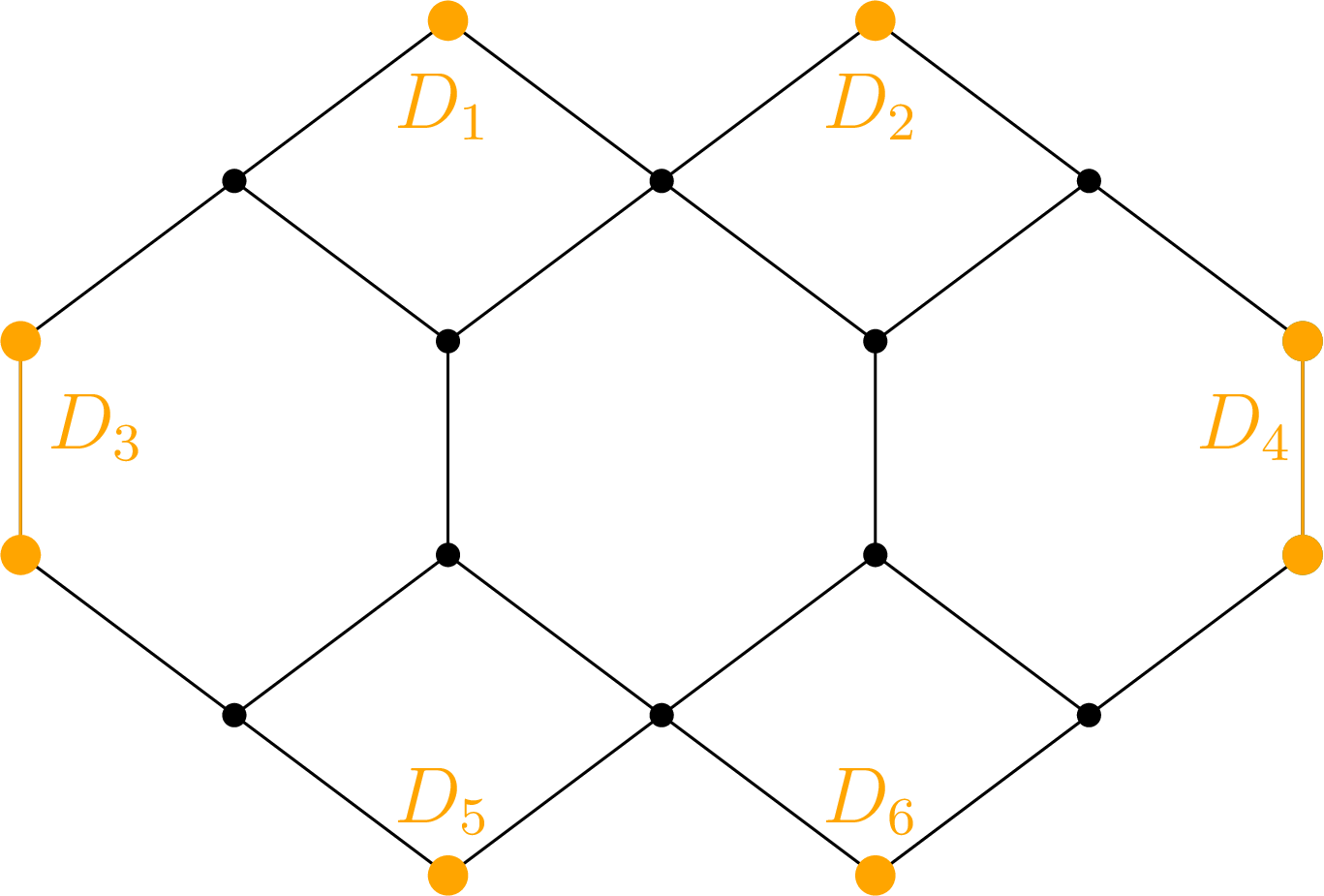}
\caption{The tope graph of a COM with VC-dimension 2 having corners $D_1,\ldots,D_6$.}
\label{fig:4}
\end{figure}

Removing a corner $D$ in a tope graph of a COM $G$ results in $G\setminus D$ also being the tope graph of a COM \cite{knauer2020corners}, where the embedding is inherited from $G$ since $G\setminus D$ is an isometric subgraph of $G$. The tope graph of a COM $G$ is said to have a \emph{corner peeling} if one can partition its vertices into $D_1,\ldots, D_k$, such that $D_i$ is a corner of $G \setminus \cup_{j=1}^{i-1} D_j$. 

We have the following corollary of Proposition \ref{prop:extend}:
\begin{cor}\label{cor:com}
Every tope graph of a COM $G$, that has a corner peeling, has a reconstructible map, thus a proper unlabeled sample compression scheme of size bounded by $\vc(G)$.
\end{cor}
\begin{proof}
Let $D$ be a corner of $G$ and inductively assume that $G\setminus D$ has a reconstructible map $\overline{a}$. Expanding $\overline{a}$ to a reconstructible map $a$ of $G$ can be done identically as in the proof of Proposition \ref{prop:extend}, by defining $a(C) = \overline{a}(C \cap (G\setminus D))$ if $C \cap (G\setminus D)$ is non-empty and $a(C) = b(C)$ otherwise, where $b$ is given by Lemma \ref{lem:fullconv} for the tope graph  $H$ of the unique sub-OM, that includes $D$. In fact, the extension is well defined since $G\setminus D$ is an isometric subgraph of $G$ and hence for every convex set $C$ of $G$, also $C \cap (G\setminus D)$ is convex in $G\setminus D$.

The only point that needs to be proved to repeat the proof of Proposition \ref{prop:extend}  is that the image of $\overline{a}$ and the image of $b$, limited to the convex sets of $D$, do not intersect. The crucial point is that for every $V=b(C), C\subseteq D$, the set $V$ is not shattered by  $G\setminus D$, so $V\neq \overline{a}(C)$ for all $C$. It follows from \cite[Lemma 5.8]{knauer2020corners}, that if $V$ is shattered by a COM $G$, it must be shattered by a sub-OM. Since $b$ maps to sets of order equal to $\vc(H)$ and $H$ is a maximal sub-OM, it is a unique sub-OM that shatters $U$. But $H$ does not exist in $G\setminus D$, more precisely $H\setminus D$ has VC-dimension less than $|V|$.
\end{proof}

The above result generalizes the result of \cite{rubinstein2012geometric} that ample classes with corner peelings have proper unlabeled sample compression schemes bounded by their VC-dimension. In fact, our result in this particular case gives the same maps as the so-called representation maps introduced in \cite{chalopin2022unlabeled}. It is known that not all ample classes have a corner peeling, although they might still have representation or reconstructible maps. On the other hand, as a consequence of our result, the hypercellular graphs \cite{Che-16} (hence also bipartite cellular graphs \cite{bandelt1996cellular}), COMs with VC-dimension at most 2 and realizable COMs all have  proper unlabeled sample compression schemes bounded by their VC-dimension, since they were proven to have a corner peeling \cite{knauer2020corners}. Moreover, by results of \cite{chepoi2022ample}, partial cubes with VC-dimension 2 can be extended to COMs with the same VC-dimension, implying that they have (improper) unlabeled compression schemes of size 2.

\section{Acknowledgment}

The research was partially supported by ARRS projects P1-0297, N1-0095, J1-1693, N1-0218, and project SiQUID funded by the Digital Europe Programme (DIGITAL).

\bibliographystyle{my-siam}
{\bibliography{compression}}

\end{document}